\documentclass[12pt, reqno]{amsart}
\usepackage{amssymb,latexsym,amsmath,amsfonts}
\usepackage{supertabular}
\usepackage{enumerate}
\usepackage[mathscr]{eucal}
\usepackage{enumitem}
\usepackage{hyperref}
\usepackage[nameinlink]{cleveref}

\allowdisplaybreaks

\numberwithin{equation}{section}
\theoremstyle{definition}
\newtheorem{definition}{Definition}[section]
\newtheorem{example}[definition]{Example}

\theoremstyle{remark}
\newtheorem{remark}[definition]{Remark}

\theoremstyle{plain}
\newtheorem{proposition}[definition]{Proposition}
\newtheorem{theorem}[definition]{Theorem}
\newtheorem{lemma}[definition]{Lemma}
\newtheorem{result}[definition]{Result}

\usepackage[usenames]{color}

\newcommand{\OM}{\Omega}

\newcommand{\unitdisk}{\mathbb{D}}

\newcommand{\clos}[1]{\overline{#1}}
\newcommand{\koba}{\mathsf{k}}

\newcommand{\lrarw}{\longrightarrow}

\newcommand{\bdy}{\partial}

\newcommand{\bcdot}{\boldsymbol{\cdot}}

\newcommand{\distance}{\mathrm{dist}}

\newcommand{\Z}{\mathbb{Z}}
\newcommand{\nat}{\mathbb{N}}

\newcommand{\C}{\mathbb{C}} 
\newcommand{\R}{\mathbb{R}} 

\makeatletter
\newcommand*{\rom}[1]{\expandafter\@slowromancap\romannumeral #1@}
\makeatother

\begin{document}
\hfill
\title[Horospheres]{Geometry of horospheres in Kobayashi hyperbolic domains}
	\author{Vikramjeet Singh Chandel, Nishith Mandal}
	
\address{Department of mathematics \& statistics, IIT Kanpur}
\email{vschandel@iitk.ac.in}
	
\address{Department of mathematics \& statistics, IIT Kanpur}
\email{nishithm21@iitk.ac.in, nishithmath@gmail.com}



\begin{abstract}
For a Kobayashi hyperbolic domain, Abate introduced the notion of small and big horospheres of 
a given radius at a boundary point with a pole. In this article, we investigate
which domains have the property that closed big horospheres and closed small horospheres centered 
at a given point and of a given radius intersect the 
boundary only at that point? We prove that any model-Gromov-hyperbolic domain have this property.
To provide examples of non-Gromov-hyperbolic domains, 
we show that unbounded locally model-Gromov-hyperbolic
domains and bounded, Dini-smooth, locally convex domains\,---\,that are locally visible\,---\,also have
this property. 
Finally, using the geometry of the horospheres, 
we present a result about the homeomorphic extension of biholomorphisms and give an application of it. 
\end{abstract}

\maketitle

\section{Introduction and statement of the results}\label{S:intro}
Let $\OM\subset\C^d$ be a Kobayashi hyperbolic domain equipped with
the Kobayashi distance $\koba_\OM$. Given a point
$x\in\bdy\OM$, $o\in\OM$ and $R>0$, Abate \cite{{Abate:horo88}} defined the following objects:
 \begin{align*}
 H^s_o(x, R)&:=\Big\{z\in\OM\,:\,\limsup_{w\to x}\big(\koba_\OM(z, w)-\koba_\OM(o, w)\big)<(1/2)\log R\Big\}\\
 H^b_o(x, R)&:=\Big\{z\in\OM\,:\,\liminf_{w\to x}\big(\koba_\OM(z, w)-\koba_\OM(o, w)\big)<(1/2)\log R\Big\}. 
 \end{align*}
The sets $H^s_o(x, R),  H^b_o(x, R)$ are called the {\em small horosphere} and the {\em big horosphere} 
with center $x$, pole $o$ and radius $R$ respectively. Some of their basic properties 
are stated in Section~\ref{S:prelim}. 
Note that although the objects $H^s_o(x, R)$, $H^b_o(x, R)$ can be defined for any Kobayashi hyperbolic domain,
they could be empty set. For $H^b_o(x, R)$ completeness of $\koba_\OM$ is a 
natural condition under which they are nonempty for all choices of parameters; see
Lemma~\ref{lm:nonemphoro} and Lemma~\ref{lm:horos_empty} in Section~\ref{s:genresulthoro}. 
On the other hand, $H^s_o(x, R)$ could be empty, for some choice of parameters, for a domain $\OM$ with 
$\koba_\OM$ complete; see e.g. Example~\ref{exa:smallhoroempty}.
\smallskip

In the case
$\OM=\unitdisk$, the open unit disc in 
the complex plane centered at $0$, using the formula for $\koba_\unitdisk$, a straightforward computation shows 
that when $o=0$, $x\in\bdy\unitdisk$, for each $z\in\unitdisk$, one has 
\[
\lim_{w\to x}\big(\koba_\unitdisk(z, w)-\koba_\unitdisk(o, w)\big)=\frac{1}{2}\log\Bigg(\frac{|x-z|^2}{1-|z|^2}\Bigg), 
\]
and therefore
\[
H^s_o(x, R)=H^b_o(x, R)=H(x, R):=\Bigg\{z\in\unitdisk\,:\,\frac{|x-z|^2}{1-|z|^2}<R\Bigg\}. 
\]
Geometrically, $H(x, R)$ is the disc with center $x/{(1+R)}$ and radius $R/(R+1)$, internally tangent to $\bdy\unitdisk$
at $x$. In particular, closures of both the small and big horospheres of arbitrary radius with 
center $x$ and pole at $0$ intersect the 
boundary of the unit disc only at the point $x$. An analogous result holds true for the Euclidean unit ball $\mathbb{B}^n$ of 
$\C^n$, $n\geq 2$. On the other hand, for the $n$-dimensional polydisc ${\unitdisk}^n$, $n\geq 2$, things are quite different. In this case, 
taking $o=0\in{\unitdisk}^n$, Abate
showed that $\clos{H^s_0(x, R)}\cap\bdy{\unitdisk}^n=\{x\}$
if and only if $x$ is on the Shilov boundary $(\bdy\unitdisk)^n$ of ${\unitdisk}^n$. On the other hand,
$\clos{H^b_0(x, R)}\cap\bdy{\unitdisk}^n$ always contains $\{x\}$ properly
(see Proposition~2.4.12 and Corollary~2.4.13 in \cite{Abate:iteration89}). 
\smallskip

Motivated from the above observations, we ask the following problem: 
\begin{itemize}
\item[$(*)$] Under what conditions on a domain $\OM\subset\C^d$, it is true that for any $x\in\bdy\OM$, $o\in\OM$ and $R>0$
\begin{equation}\label{E:smallbigsingle}
  \clos{H^s_o(x, R)}\cap\bdy{\OM}=\clos{H^b_o(x, R)}\cap\bdy{\OM}=\{x\}?
  \end{equation}
\end{itemize}
Geometrically, condition~\ref{E:smallbigsingle} above implies that both $H^s_o(x, R)$ and $H^b_o(x, R)$ should be 
nonempty and their diameter tends to zero as $R$ decreases to zero. 
If we restrict our attention to bounded simply-connected planar domains, to the best of our knowledge, we 
are not aware of any result that characterizes such domains having property \ref{E:smallbigsingle} above.
We present such a characterization.
\begin{theorem}\label{thm:jordandomain}
Let $\OM\subset\C$ be a bounded simply-connected domain. Let 
$\psi:\unitdisk\lrarw\OM$ be a Riemann map and let $o=\psi(0)$. Then the following are equivalent.
\begin{itemize}
\item For any $x\in\bdy\OM$ and for any $R>0$, $\clos{H^s_o(x, R)}\cap\bdy\OM=\clos{H^b_o(x, R)}\cap\bdy\OM=\{x\}$. 
\smallskip

\item The map $\psi:\unitdisk\lrarw\OM$ extends as a homeomorphism from $\clos{\unitdisk}$ onto $\clos{\OM}$, i.e., 
$\OM$ is a Jordan domain. 
\end{itemize}
\end{theorem}
\noindent An upshot of the above theorem is that for bounded planar simply-connected domains condition~\ref{E:smallbigsingle} above
implies that the small horospheres and big horospheres with the same parameters coincide. Theorem~\ref{thm:jordandomain}
is a consequence of one of the main results of this paper, namely Theorem~\ref{thm:contextbiholo}, regarding the homeomorphic 
extension of a biholomorphism between bounded domains.
%where one of the domains satisfies condition~\ref{E:smallbigsingle} above. 
This later result gives another reason of studying domains having
condition~\ref{E:smallbigsingle}.
\smallskip

In higher dimensions, Abate first proved that for a bounded strongly pseudoconvex domain
property \eqref{E:smallbigsingle} holds true; see \cite[Theorem~2.4.14]{Abate:iteration89}. 
In fact, using Lempert’s theory of complex geodesics \cite{LL:1981}, Abate also showed that if the boundary of a bounded 
strongly convex domain is at least $\mathcal{C}^3$, then, for each $z\in\OM$, the limit 
\begin{equation}\label{E:limhoro}
 \lim_{w\to x}\big(\koba_\OM(z, w)-\koba_\OM(o, w)\big)
 \end{equation}
exists and thus as a consequence big and small horospheres coincide (see \cite{Abate:scd90}).
\smallskip

To understand the existence of limit in \eqref{E:limhoro}, Arosio et al. in \cite{AFSG:2024} considered the
{\em horofunction compactification} $\clos{\OM}^{H}$ and the {\em horofunction boundary} $\bdy_H\OM$. 
These notions were introduced by Gromov in 1981 \cite{Gromov:1981}. Later, Marc Rieffel rediscovered this compactification as the 
maximal ideal space of a unital cummutative $C^{*}$-algebra. He refers to it as the {\em metric compactification} and
the corresponding boundary as the {\em metric boundary}; see \cite[Section~4]{Rieffel:2002}. From this later point of view, 
for a bounded domain $\OM$, the existence of the 
limit in \eqref{E:limhoro} is equivalent to the existence of a continuous map from $\clos{\OM}$ to $\clos{\OM}^{H}$ 
extending the identity map on $\OM$. Working in this setting, Arosio et al. proved that if $\OM$ is either
\begin{itemize}
\item a bounded strongly pseudoconvex domain with $\mathcal{C}^2$-boundary in $\C^d$,
\item a bounded convex domain such that $s_\OM(z)\to1$ as $z\to\bdy\OM$, 
\item a bounded convex domain in $\C^d$ with $\mathcal{C}^{\infty}$-boundary, of finite type in the sense of D’Angelo, 
\end{itemize}
then the limit in \eqref{E:limhoro} exists. Here, $s_\OM(z)$ denotes the squeezing function of the domain; see \cite{DGZ:2012} for the 
definition and basic properties. 
\smallskip

The above three types of domains are {\em model-Gromov-hyperbolic domains}, i.e., $(\OM, \koba_\OM)$ is Gromov 
hyperbolic and the identity map on $\OM$ extends as a homeomorphism from the Gromov compactification $\clos{\OM}^{G}$
to $\clos{\OM}$; this follows from the works of Balogh--Bonk \cite{BB:2000}, Zimmer \cite{{Zimmer:sub2022}, {Zimmer:cdft2016}}, 
together with the results in \cite{BGZ:2021}. Arosio et al. observed that in the 
above three cases the distance space $(\OM, \koba_\OM)$ has an additional property, namely, the {\em approaching geodesic property}. 
This property gives a sufficient condition under which the identity map on $\OM$ extends as a homeomorphism from 
$\clos{\OM}^{G}$ to $\clos{\OM}^{H}$, i.e., any two of the three compactifications $\clos{\OM}$, 
$\clos{\OM}^{G}$, $\clos{\OM}^{H}$ admits homeomorphisms extending the identity map.
\smallskip
 
Recently, Bharali--Zimmer \cite[Theorem~1.11]{BZ:2023} showed that a 
model-Gromov-hyperbolic domain is Cauchy complete and satisfies
visibility property with respect to the geodesics of the Kobayashi distance,
see \ref{ss:visibility} for the definition of this property. 
In this paper, we show that
\begin{proposition}\label{P:visi_horo}
Let $\OM$ be a complete Kobayashi hyperbolic domain that satisfies the visibility property with respect to the geodesics 
of the Kobayashi distance. Then for any $o\in\OM$, $x\in\bdy\OM$ and $R>0$ we have 
\[ 
  \overline{H_o^b(x, R)}\cap\bdy{\OM}=\{x\}. 
          \]
\end{proposition}
\noindent As an upshot of the above result and the existence of limit in \eqref{E:limhoro}, for the three
classes of domains considered by Arosio et al., property \eqref{E:smallbigsingle} in the problem $(*)$ holds true. However, whether 
\eqref{E:smallbigsingle} holds true 
for a {\em general} model-Gromov-hyperbolic domain is not known. In this article,
we answer this question affirmatively. Before we state this result, a notation: since we would also like to consider unbounded 
domains, the natural formulation of model-Gromov-hyperbolic domains is via the end-point compactification $\clos{\OM}^{End}$.
For bounded 
domains $\clos{\OM}^{End}=\clos{\OM}$; see \cite{CGMS:2024} for more details about end-point compactification. 
\begin{theorem}\label{thm:gromovmodel}
Let $\OM$ be a Kobayashi hyperbolic domain and suppose $(\OM, \koba_\OM)$ is Gromov hyperbolic metric space. If
the identity map ${\sf id}_\OM$ extends to a homeomorphism from $\clos{\OM}^{G}$ onto $\clos{\OM}^{End}$ then
for any $x\in\bdy{\OM}$, $o\in\OM$ and $R>0$, we have: 
\[
  \clos{H^s_o(x, R)}\cap\bdy\OM= \clos{H^b_o(x, R)}\cap\bdy\OM=\{x\}.
    \]
      \end{theorem}
\noindent 
Owing to the aforementioned result of Bharali--Zimmer and Proposition~\ref{P:visi_horo}, we know that 
$\clos{H^b_o(x, R)}\cap\bdy\OM=\{x\}$ for every $x\in\bdy\OM$, for a model-Gromov-hyperbolic domain. 
Gromov hyperbolicity plays a crucial role in showing that
$x\in\clos{H^s_o(x, R)}$. 
More specifically, it is the property of uniform separation of 
geodesic rays starting at a fixed point in the domain and landing at the same point in the Gromov boundary that is at the heart of our proof
that $\clos{H^s_o(x, R)}\cap\bdy\OM=\{x\}$. 
There is no complete characterization of a model-Gromov-hyperbolic domain. Except the three
considered by Arosio et al., a few others are also listed in \cite{BGNT:2024}. 
\smallskip

\begin{remark}\label{Rem:altproof}             
We also present an 
alternative proof\,---\,suggested by the referee\,---\,that does not rely on visibility.
This second proof draws upon the horofunction compactification,
the correspondence between the horofunction and Gromov boundaries together with 
the uniform boundedness of horofunctions corresponding to the same Gromov boundary point.
The two proofs emphasize different geometric 
aspects, namely\,---\,visibility versus the horofunction
compactification.
\end{remark}

A natural question arises whether there are non-Gromov-hyperbolic domains 
having property \eqref{E:smallbigsingle} in the problem $(*)$. We provide two 
results in this direction. Our first result deals with domains that are {\em locally model-Gromov-hyperbolic domain}; 
see Section~\ref{S:proofgromovmodel} for the definition. 
\begin{theorem}\label{thm:locmodgrom}
Let $\OM\subset\C^d$ be an complete Kobayashi hyperbolic domain
that is hyperbolically embedded in $\C^d$. Suppose $\OM$ is locally model-Gromov-hyperbolic 
domain. Then for any $x\in\bdy\OM$, $o\in\OM$ and $R>0$, we have 
 \[
    \clos{H^s_o(x, R)}\cap\bdy{\OM}=\clos{H^b_o(x, R)}\cap\bdy{\OM}=\{x\}. 
      \]
 \end{theorem}
\noindent We refer the reader to Subsection~\ref{ss:visibility} for the 
definition of a domain being hyperbolically  embedded in $\C^d$. 
A recent result of Bracci et al. \cite{BGNT:2024} says that any bounded local model-Gromov-hyperbolic 
domain is itself a model-Gromov-hyperbolic domain. 
In Section~\ref{S:examples}, we shall provide a class of planar domains,
introduced in \cite{CGMS:2024}, that are
locally model-Gromov-hyperbolic, however, in general, they 
are not Gromov hyperbolic domains. 
\smallskip

The other class of domains that satisfy condition~\eqref{E:smallbigsingle} and provides examples of non-Gromov-hyperbolic {\em bounded} domains 
are modelled locally on bounded convex domains with Dini-smooth boundary, see \cite{NA:2017} for
the definition. To state the result we need a definition.  
Let $\OM$ be a domain in $\C^d$. We say $\OM$ is {\em locally convex} if for 
any point $x\in\bdy\OM$, there exist $\C^d$-neighbourhoods $U_x$, $V_x$ of $x$ such that 
$V_x\Subset U_x$, and an injective holomorphic map $\Phi_x:U_x\lrarw\C^d$ such that 
$\Phi_x(\OM\cap V_x)$ is an open convex set. 
\begin{theorem}\label{thm:loccondini}
Let $\OM$ be a complete Kobayashi hyperbolic bounded domain with Dini-smooth boundary.
Suppose $\OM$ is locally convex. 
Moreover, suppose for any point $x\in\bdy\OM$, the domain $\Phi_x(V_x\cap\OM)$ has
visibility property with respect to the geodesics, where $(\Phi_x, U_x, V_x)$ be as in the definition of local convexity.
Then for any $x\in\OM$, $o\in\OM$ and $R>0$, we have 
 \[
   \clos{H^s_o(x, R)}\cap\bdy{\OM}=\clos{H^b_o(x, R)}\cap\bdy{\OM}=\{x\}. 
     \]
       \end{theorem}
\noindent A bounded convex domain with smooth boundary having all the points in its boundary\,---\,except a nonempty finite subset\,---\,of 
finite D'Angelo type is a visibility domain, see \cite[Theorem~1.1]{BNT:2022}, \cite[Corollary~1.4]{CMS:2024}. 
So the above theorem applies to it. Note that by a result of 
Zimmer \cite{Zimmer:cdft2016}, such a domain is not Gromov hyperbolic. Bharali--Zimmer \cite{BZ:2017}
also constructed bounded convex domains\,---\,that are
not Gromov hyperbolic\,---\,and has visibility property. The above theorem applies to such domains too. 
A recent result of Bracci--Nikolov--Thomas says that any complete Kobayashi hyperbolic bounded 
domain with Dini-smooth boundary such that all its boundary points are $\C$-strictly convex is a visibility domain;
see \cite[Theorem~1.3]{BNT:2022}. Theorem~\ref{thm:loccondini} applies to all such domains too. 
\smallskip

The domains in Theorem~\ref{thm:gromovmodel}, Theorem~\ref{thm:locmodgrom} and 
Theorem~\ref{thm:loccondini} are all visible domains. 
It is natural to ask if the visibility property itself implies the 
property \eqref{E:smallbigsingle} in problem $(*)$. This is not true, e.g., 
the domain  $\OM:=\unitdisk\setminus\{[0, 1)\}$ is a visibility domain
but for any $x\in (0, 1]$, there exists an $R_0$, depending on $x$, such that for any $R<R_0$, we have 
\[
  H^s_o(x, R)=\emptyset \ \ \forall R<R_0
  \]
and for some $o\in\OM$; see Example~\ref{exa:smallhoroempty} for the details. 
\smallskip

We now present our result regarding the continuous extension
of biholomorphisms. 
%As an upshot, we shall characterize those bounded simply-connected domains
%in $\C$ that have property \eqref{E:smallbigsingle} in problem $(*)$. 
To present our results, we define 
a class of `nice domains'. 
%The definition of these domains is motivated from the three types of domains that 
%Arosio et al. considered in their paper.
Given a complete Kobayashi hyperbolic domain $D$, let $\xi\in\bdy{D}$ and $o\in D$ be such that 
for any $z\in D$ the limit 
\[
  \lim_{w\to\xi}\big(\koba_D(z, w)-\koba_D(o, w)\big) \ \ \text{exists}. 
 \]
 Then, for all $R>0$, we define the {\em horosphere} centered at $\xi$ of radius $R$ and pole $o$ as
 \[
   H_{o}(\xi, R):=\big\{z\in D\,:\,\lim_{w\to\xi}\big(\koba_D(z, w)-\koba_D(o, w)\big)<(1/2)\log R\big\}.
   \]
When the context is clear, we shall write $H_{o}(\xi, R)$ simply as $H(\xi, R)$. 

\begin{definition}\label{D:metreg}
A complete Kobayashi hyperbolic domain $D$ is said to be {\em metrically-regular} if there exists a point 
$o\in D$ such that 
\begin{enumerate}
\item for any $\xi\in\bdy D$ and $z\in D$, the limit 
\[
  \lim_{w\to\xi}\big(\koba_D(z, w)-\koba_D(o, w)\big) \ \ \text{exists and}
 \]

\item given two distinct points $\xi_1\neq \xi_2\in\bdy{D}$, there exists an $R>0$ such that
\[ 
  H(\xi_1, R)\cap H(\xi_2, R)=\emptyset.
  \]
\end{enumerate}
\end{definition}
\noindent The Euclidean unit ball $\mathbb{B}^d\subset\C^d$, $d\geq 1$, is metrically-regular as one explicitly
knows the formulae for the Kobayashi distance. On the other hand, $\unitdisk^d$, $d\geq 2$, is not
metrically-regular. The three types of domains considered by Arosio et al. are metrically-regular. 
\begin{remark}\label{Rm:metreg}
It is evident that when $\OM$ is bounded, the metric regularity of $\OM$
implies that the identity map on $\OM$ extends as a homeomorphism
from its horofunction compactification $\clos{\OM}^H$ 
to its Euclidean closure $\clos{\OM}$. 
\end{remark}
\noindent We are now ready to present our main result. 
\begin{theorem}\label{thm:contextbiholo}
Let $D$, $\OM$ be bounded domains
that are complete with respect to the Kobayashi distance. 
Suppose $D$ is metrically-regular and $\OM$ has the property that
\begin{itemize}
\item[$a)$] for each $x\in\bdy\OM$ and for every $R>0$, $H^s_p(x, R)$ is nonempty
with respect to some (and hence any) pole $p$, and 
\smallskip

\item[$b)$] for any two distinct points $x, y\in\bdy{\OM}$ there is an $R>0$ such that  $H^b_p(x, R)\neq H^b_p(y, R)$. 
\end{itemize}
Then any biholomorphism $\Psi:D\lrarw\OM$ extends as a homeomorphism 
from $\clos{D}$ onto $\clos{\OM}$. 
\end{theorem}
 \begin{remark}\label{rm:applicability}
Note that any bounded domain that satisfies condition~\eqref{E:smallbigsingle} in problem $(*)$ 
satisfies the condition $(a)$ and $(b)$ imposed on the domain $\OM$ in the statement of 
the above theorem. In particular, taking $D=\unitdisk$ and $\OM$ to be a bounded simply-connected 
domain in $\C$ that satisfies \eqref{E:smallbigsingle} we recover Theorem~\ref{thm:jordandomain}.
Theorem~\ref{thm:contextbiholo} also implies that bounded simply-connected domains that are metrically-regular are 
the Jordan domains.
\end{remark}

\section{Preliminaries}\label{S:prelim}
\subsection{Basic properties of horospheres.}
We begin this section with a list of important properties of the horospheres due to Abate, see 
[Section~2.4.2]\cite{Abate:iteration89}. 
In what follows, $B_\koba(o, r)$ will denote the ball with respect to the Kobayashi distance centered at $o$ of radius $r>0$.
\begin{result}[paraphrasing Lemma~2.4.10 in \cite{Abate:iteration89}]\label{res:horo_prop}
Let $\Omega$ be a Kobayashi hyperbolic domain. Then for any $o, p\in\OM$, $x\in\bdy{\OM}$, we have
\begin{enumerate}
\item for every $R>0$ we have $H^s_o(x, R)\subset H^b_o(x, R)$;
\smallskip

\item for every $0<R_1<R_2$, $H^s_o(x, R_1)\subset H^s_o(x, R_2)$ and 
$H^b_o(x, R_1)\subset H^b_o(x, R_2)$; 
\smallskip

\item for every $R>1$, we have $B_\koba(o, (1/2)\log R)\subset H^s_o(x, R)$; 
\smallskip

\item for every $R<1$ we have $H^b_o(x, R)\cap B_\koba(o, (-1/2)\log R)=\emptyset$;
\smallskip

\item $\bigcup_{R>0}H^s_o(x, R)=\bigcup_{R>0}H^b_o(x, R)=\OM$ and
$\bigcap_{R>0}H^s_o(x, R)=\bigcap_{R>0}H^b_o(x, R)=\emptyset$.
\smallskip

\item for any $R>0$, we have ${H_p^b(x, R)}\subset {H_o^b(x, LR)}$, 
${H_p^s(x, R)}\subset {H_o^s(x, LR)}$ where 
\[
({1}/{2})\log L:=\limsup_{w\to x}\big(\koba_\OM(p, w)-\koba_\OM(o, w)\big).
\]

\end{enumerate}
\end{result}
\noindent When there are more than one domains under consideration, we shall use $H^s_o(x, R, \OM)$, $H^b_o(x, R, \OM)$
to specify the domain in which horospheres are considered. 

\subsection{Results regarding visibility property of a domain}\label{ss:visibility}
In this subsection, we put together certain results regarding 
visibility property of a Kobayashi hyperbolic domain. 
\smallskip

\begin{definition}\label{def:geo-vis}
A complete Kobayashi hyperbolic domain $\OM\subset\C^d$ is
said to possess visibility property with respect to the geodesics
of the Kobayashi distance $\koba_\OM$ if for any two distinct points $x$, $y\in \bdy{\OM}$,
there exist disjoint $\C^d$-neighbourhoods $U_x$, $U_y$ of $x$ and $y$ respectively, 
and a compact subset $K\subset\OM$ 
such that any geodesic that starts in $U_x$ and ends in $U_y$ intersects $K$. 
\end{definition}

In this paper, we shall mostly be concerned with this form of visibility which is designed for complete spaces. 
There are two other notions of visibility, namely weak visibility and visibility that are designed 
considering that, in general, it is difficult to determine whether $(\OM, \koba_\OM)$ is complete. 
The definition of weak visibility is analogous to Definition~\ref{def:geo-vis} except that the geodesics 
are replaced by $(1, \kappa)$-almost geodesics for all $\kappa>0$. For more information, 
we refer the reader to \cite{BZ:2017}, \cite{BNT:2022}, \cite{BZ:2023}. 
For complete hyperbolic domains, we have the following:
\begin{result}[{\cite[Proposition~2.5]{BNT:2022}, \cite[Proposition~3.1]{CMS:2024}}] \label{res:vis-weakvis}
Let $\OM$ be a complete Kobayashi hyperbolic domain in $\C^d$. Then the following are equivalent. 
\begin{itemize}
\item $\OM$ has visibility property with respect to the geodesics. 
\smallskip 

\item $\OM$ has weak visibility property. 
\smallskip

\item For any $o\in\OM$, and for any $x\neq y\in\bdy\OM$ we have
\[
 \limsup_{z\to x,\,w\to y}\langle z\,|\,w\rangle^{\OM}_o<\infty.
 \]
\end{itemize}
\end{result}
\noindent Here, $\langle z\,|\,w\rangle^{\OM}_o$ denotes the Gromov product of $z$ and $w$ with respect to $o$ defined by
\[
 \langle z\,|\,w\rangle^{\OM}_o:=\frac{\koba_\OM(o, z)+\koba_\OM(o, w)-\koba_\OM(z, w)}{2}. 
 \]
The above result was established for bounded domains. However, it can also be 
proved for unbounded domains using a `localization trick' provided 
by \cite[Lemma~2.2]{CGMS:2024} to prove \cite[Theorem~1.3]{CGMS:2024}. 
\smallskip

A Kobayashi hyperbolic domain $\OM\subset\C^d$ is called {\em hyperbolically embedded} in $\C^d$ 
 if for any two distinct points $x\neq y\in\bdy\OM$, we have 
 \[
 \liminf_{z\to x,\,w\to y}\koba_\OM(z, w)>0.
 \]
 It is a fact that all planar domains are hyperbolically embedded in $\C$. A bounded 
 domain $\OM\subset\C^d$ is hyperbolically embedded in $\C^d$. We now state the next result. 
 \begin{result}[{\cite[Theorem~1.6]{CGMS:2024}, \cite[Theorem~1.6]{RM:2024}}]\label{res:locvis-globvis}
 Let $\OM\subset\C^d$ be a Kobayashi hyperbolic domain that is hyperbolically embedded in $\C^d$. 
 Suppose for any $x\in\bdy\OM$, there exists a $\C^d$-neighbourhood $U_x$ of $x$ such that
 $U_x\cap\OM$ is connected and $(U_x\cap\OM, \koba_{U_x\cap\OM})$ has weak visibility 
 property. Then $(\OM, \koba_\OM)$ has weak visibility property. 
 \end{result}
\noindent The last result that we state is a localization result for the Kobayashi distance in weak visibility domains. 
\begin{result}[Corollary of Theorem~1.3 in \cite{Amar:2023}]\label{res:localization}
Let $\OM$ be a Kobayashi hyperbolic domain that has weak visibility 
property. Suppose for any point $x\in\bdy{\OM}$, there exist 
$\C^d$-neighbourhoods $U_x$, $W_x$ of $x$ such that 
$W_x\Subset U_x$ and $\OM\cap U_x$ is connected. Then 
there exists a $C>0$ such that 
\[
 \koba_{U_x\cap\OM}(z, w)\leq \koba_\OM(z,w)+C, \ \ \forall z, w\in 
 W_x\cap\OM.
\]
\end{result}

\section{General results about horospheres}\label{s:genresulthoro}
In this section, we present some general results about horospheres. At the end of this section, 
we shall present the proof of Proposition~\ref{P:visi_horo}. 
We begin with an elementary lemma.
\begin{lemma}\label{lm:horos_empty}
Let $\OM$ be a hyperbolic domain and let $o\in\OM$ and $x\in\bdy{\OM}$ be given.
Suppose that
\[ 
M:= \limsup_{w\to x} \koba_\OM(o, w)<\infty.
 \]
Then for any $0<R<e^{-2M}$, we have $H^s_o(x, R)\subset H^b_o(x, R)=\emptyset$. 
\end{lemma}
\begin{proof}
For a fixed $z\in\OM$, we have
$\koba_\OM(z, w)-\koba_\OM(o, w)\geq -\koba_\OM(o, w)$. 
Therefore 
\[
\liminf_{w\to x}\big(\koba_\OM(z, w)-\koba_\OM(o, w)\big)\geq\liminf_{w\to x}\big(-\koba_\OM(o, w)\big)
  =-\limsup_{w\to x}\koba_\OM(o, w)=-M
\]
Now if $z\in H^b_o(x, R)$ for some $R>0$, then the above inequality and the definition of
$H^b_o(x, R)$
implies
\[
M+{1}/{2}\log R>0 \Longleftrightarrow R> e^{-2M}. 
 \]
This implies that for any $R<e^{-2M}$ the big horosphere $H^b_o(x, R)=\emptyset$.
\end{proof}
\noindent The above lemma suggests that to study the geometry of the horospheres with respect to a pole
$o\in\OM$, we should impose on $\OM$ the following condition:
\[
 \limsup_{w\to x}\koba_\OM(o, w)=\infty \ \ \text{$\forall x\in\bdy{\OM}$}.
 \]
Note that every complete Kobayashi hyperbolic domain satisfies this property.
Indeed for such domains we have the following:
\begin{lemma}\label{lm:nonemphoro}
Let $\OM$ be a complete Kobayashi hyperbolic domain in $\C^d$. Then for any $o\in\OM$, $x\in\bdy{\OM}$ and $R>0$, the set 
$H^b_o(x, R)$ is non-empty. Moreover, 
if $\OM$ is bounded, then the set $\clos{H^b_o(x, R)}\cap\bdy{\OM}$ is also non-empty. 
\end{lemma}
\begin{proof}
Fix $o\in\OM$, $x\in\bdy{\OM}$ and let $R>0$ be given. 
Choose a sequence $(w_n)\subset\OM$ such that $w_n\to x$ as $n\to\infty$ and, 
for a given $s>0$, consider the ball $B_\koba(o, s)$. 
Since $(\OM, \koba_\OM)$ is a locally compact length metric space that is complete, 
by Hopf--Rinow Theorem, every bounded set is relatively compact. 
Therefore, without loss of generality, we can assume that $w_n\in\OM\setminus B_\koba(o, s)$.
For any $n$, choose $b_n\in\bdy B_\koba(o, s)$ such that 
\begin{equation}\label{E:geod}
    \koba_\OM(o,w_n)=\koba_\OM(o,b_n)+\koba_\OM(b_n, w_n). 
     \end{equation}
By passing to subsequences, we may assume that $b_n\to b_s$ as $n\to\infty$ for some $b_s\in\bdy B_\koba(o, s)$. 
Therefore, given $\epsilon>0$ there exists $n_0\in\nat$ such that 
\begin{equation*}
    \koba_\OM(b_s, b_n)<\epsilon \ \ \ \forall n\geq n_0. 
     \end{equation*}
\textbf{Claim.} $b_s\in H^b_o(x, e^{-s})$.
\smallskip

\noindent Given $\epsilon>0$
the triangle inequality and the inequality above implies
\begin{equation*}
    \begin{split}
        \koba_\OM(b_s,w_n) &\leq\koba_\OM(b_s,b_n)+\koba_\OM(b_n,w_n)\\
        &<\epsilon +\koba_\OM(b_n,w_n)\hspace{3mm} \forall n\geq n_0.
    \end{split}
\end{equation*}
The above inequality and \eqref{E:geod} implies that, for all $n\geq n_0$, we have 
    \[
    \koba_\OM(b_s, w_n)-\koba_\OM(o, w_n)
      <-s+\epsilon. 
    \]
Therefore, for every $\epsilon>0$, we have
\[\liminf_{w\to x}\big(\koba_\OM(b_s, w)-\koba_\OM(o, w)\big)\leq 
\liminf_{n\to\infty}\big(\koba_\OM(b_s, w_n)-\koba_\OM(o, w_n)\big)\leq-s+\epsilon.
 \]
This implies that 
\[
  \liminf_{w\to x}\big(\koba_\OM(b_s, w)-\koba_\OM(o, w)\big)<({1}/{2})\log{e}^{-s}. 
  \]
  whence the claim.\hfill $\blacktriangleleft$
\smallskip
  
It is enough to show that $H^b_o(x, R)$ is non-empty for all $R<1$. Fix such an $R$, then 
for every $s\in(-\log{R}, \infty)$, there exists a point $b_s$ such that $\koba_\OM(o, b_s)=s$ 
and $b_s\in H^b_o(x, e^{-s})\subset H^b_o(x, R)$. This implies that
$H^b_o(x, R)$ is nonempty. If $\OM$ is bounded, since $\lim_{s\to\infty}\koba_\OM(o, b_s)=\infty$, 
the set $\clos{H^b_o(x, R)}\cap\bdy{\OM}$ is nonempty too, for any $R<1$, and hence 
for any $R>0$. 
\end{proof}
The above lemma does not give any information about the points in the set $\overline{H^b(x, R)}\cap\bdy{\OM}$. 
Our next lemma is a result in this direction. 

\begin{lemma}\label{lm:clusgeodray}
Let $\OM$ be a complete Kobayashi hyperbolic domain and let 
$\gamma:[0, \infty)\longrightarrow(\OM, \koba_\OM)$ be a geodesic ray. Suppose $x, y\in\bdy{\OM}$ be such that there 
exist sequences $(t_n)$ and $(s_n)$ such that $t_n\to\infty$, $s_n\to\infty$ and 
\[
  x=\lim_{n\to\infty}\gamma(t_n), \ \ \   y=\lim_{n\to\infty}\gamma(s_n). 
  \]
  Then, for all $R>0$, $o\in\OM$, $y\in\overline{H^b_o(x, R)}$; in particular, taking $s_n=t_n$, $x\in\overline{H^b_o(x, R)}$. 
  \end{lemma}
  \begin{proof}
  Let $p=\gamma(0)$.
  \smallskip
  
  \noindent{\bf Claim.} For all $R>0$, $y\in\overline{H_p^b(x, R)}$. 
  \smallskip
  
  \noindent Suppose $y\notin\overline{H_p^b(x, R)}$ for some $R>0$. Then there exists an $r>0$ such that for any $z\in B(y, r)\cap\OM$
  we have 
  \[ 
   \liminf_{w\to x}\big(\koba_\OM(z, w)-\koba_\OM(p, w)\big)\geq (1/2)\log R.
   \]
    Since $\gamma(s_n)\to y$, there exists $N_0\in\nat$ such that for all $n\geq N_0$, $\gamma(s_n)\in B(y, r)$. Fix a $k\geq N_0$, then as $\gamma(t_n)\to x$, we have
    \begin{align*}
     \liminf_{n\to\infty}\big(\koba_\OM(\gamma(s_k), \gamma(t_n))-\koba_\OM(\gamma(0), \gamma(t_n))\big)&\geq
     \liminf_{w\to x}\big(\koba_\OM(\gamma(s_k), w)-\koba_\OM(\gamma(0), w)\big)\\
     &\geq(1/2)\log R. 
     \end{align*}
     Since $\gamma$ is a geodesic ray, the above inequality implies that
     \[
      s_k+(1/2)\log R\leq 0 \ \ \ \forall k\geq N_0, 
      \]
      which is a contradiction. Hence the claim.\hfill$\blacktriangleleft$
      \smallskip
      
      By Lemma~\ref{res:horo_prop}, part~$(6)$, for all $S>0$, we have ${H_p^b(x, S)}\subset {H_o^b(x, LS)}$
      where $L$ is given by the identity 
       \[
        ({1}/{2})\log L=\limsup_{w\to x}\big(\koba_\OM(p, w)-\koba_\OM(o, w)\big).  
        \]
        Clearly if $L\leq1$, then we are done, otherwise for a given $R$ we choose $S>0$ small enough such that
        $LS<R$, since $y\in\overline{H_p^b(x, S)}$, the inclusion above implies the desired conclusion. 
  \end{proof}
  %\begin{remark}\label{rm:clusterset}
   \noindent Given a geodesic ray $\gamma:[0, \infty)\lrarw(\OM, \koba_\OM)$, consider the {\em cluster set of $\gamma$}
  \[
   \mathcal{C}(\gamma):=\big\{x\in\bdy{\OM}\,:\,\text{$\exists$ a sequence $(t_n)$, $t_n\to\infty$, such that $\gamma(t_n)\to x$}\big\}. 
   \]
   Lemma~\ref{lm:clusgeodray} says that for any $x\in\mathcal{C}(\gamma)$, $\mathcal{C}(\gamma)\subset\overline{H^b_o(x, R)}$
   for all $R>0$. In particular, given an $x\in\bdy{\OM}$, if there is a geodesic ray $\gamma$ such that 
   $x\in\mathcal{C}(\gamma)$ then $x\in\overline{H^b_o(x, R)}$. Whether the later property holds true for a
   general complete Kobayashi hyperbolic domain $\OM$ is not clear. 
   \smallskip

   Lemma~\ref{lm:clusgeodray} also helps in 
   understanding why for the bidisc ${\unitdisk}^2$, given $x=(x_1, x_2)\in\bdy{{\unitdisk}^2}$, and $R>0$,
   $H^b_o(x, R)$, with pole $0$, contains $x$ properly\,---\,as observed
   by Abate by direct computation. To see how, assuming $x_1\in\bdy\unitdisk$, we consider a geodesic ray 
   of ${\unitdisk}^2$ of the form $\gamma:=(\sigma, f\circ\sigma)$, 
   where $\sigma$ is the radial geodesic ray of the unit disc starting at $0\in\unitdisk$ and landing at
   $x_1$ and $f$ is a holomorphic self-map of $\unitdisk$ such that the radial cluster set of $f$ at $x_1$ contains $x_2$ properly. It is a non-trivial fact that such an $f$ exist. Clearly then 
   $\mathcal{C}(\gamma)$ contains $(x_1, x_2)$ properly and then by the above 
   observation $H^b_o(x, R)$ contains
   $(x_1, x_2)$ properly. 
   \smallskip
   
   For Kobayashi hyperbolic convex 
   domains, we show that the closed big horospheres centered at a point contain that point. For this, we need a result due to Zimmer.
   \begin{result}[{paraphrasing Zimmer \cite[Lemma~3.2]{Zimmer:cdft2016}}]\label{Res:con_quasi}
   Let $\OM$ be a hyperbolic convex domain and let $o\in\OM$. Given a boundary point $x\in\bdy{\OM}$, there exists 
   $\alpha\geq 1$ and $\beta\geq 0$ such that the curve $\sigma_x:[0, \infty)\lrarw\Omega$ given by
   \[
     \sigma_x(t)=x+e^{-2t}(o-x)
     \]
     is an $(\alpha, \beta)$-quasi-geodesic. 
     \end{result}
\noindent We now present the result alluded to as above. 
     \begin{proposition}\label{prop:convex}
     Let $\Omega$ be a Kobayashi hyperbolic convex domain and let $x\in\bdy\OM$. 
     Then for any $R>0$ and any $o\in\Omega$, we have
     \[
     x\in\overline{H_o^b(x, R)}.
     \]
     \end{proposition}
     \begin{proof} 
      Suppose $x\notin\overline{{H_o^b(x, R)}}$ for some $R$ and $o\in\OM$. Clearly, 
      we can find an $r_0>0$ such that 
      \[
       B(x, r_0)\cap\OM\cap\overline{H_o^b(x, R)}=\emptyset.
       \]
       Therefore, given $\epsilon>0$, for any $z\in B(x, r_0)\cap\OM$, we have 
       \[
         \liminf_{w\to x}\big(\koba_\OM(z, w)-\koba_\OM(o, w)\big)\geq(1/2)\log R>(1/2)\log R-\epsilon. 
         \]
      This implies that given $\epsilon>0$, 
      and $z\in B(x, r_0)\cap\Omega$, we can find a neighbourhood $U(z, \epsilon)$ of $x$ such that
      \[
      \inf_{w\in U(z, \epsilon)}\big(\koba_\OM(z, w)-\koba_\OM(o, w)\big)>({1}/{2})\log R-\epsilon
      \]
      which is equivalent to 
      \begin{equation}\label{E:xnotinhoro}
       \koba_\OM(z, w)-\koba_\OM(o, w)>({1}/{2})\log R-\epsilon \ \ \forall w\in U(z, \epsilon). 
       \end{equation}
       Let $\sigma_x:[0, \infty)\lrarw\Omega$ be an $(\alpha, \beta)$-quasi-geodesic as given by Result~\ref{Res:con_quasi}. 
       In particular, if $z\in{\rm Range}({\sigma_x})\cap B(x, r_0)$ and $w\in{\rm Range}({\sigma_x})\cap U(z, \epsilon)$, 
       the above inequality 
       is satisfied. Let $t_2>>t_1$ be such that $\sigma_x(t_1)\in B(x, r_0)$ and $\sigma_x(t_2)\in U(z, \epsilon)$. Then
       \[
       \koba_\OM(\sigma_x(t_1), \sigma_x(t_2))\leq\alpha(t_2-t_1)+\beta\ \ \text{and} \ \ \koba_\OM(o, \sigma_x(t_2))\geq \alpha t_2-\beta. 
       \]
       This implies that 
       \[
        \koba_\OM(z, w)-\koba_\OM(o, w)\leq 2\beta-\alpha t_1.
        \]
        Therefore, from \eqref{E:xnotinhoro}, it follows that 
        \[
         2\beta-\alpha t_1>({1}/{2})\log R-\epsilon \ \ \text{$\forall t_1$ such that $\sigma_x(t_1)\in B(x, r_0)$}.
         \]
         Since the set of all such $t_1$ is unbounded, we arrive at a contradiction, whence the result. 
         \end{proof}

         We shall now present the proof of Proposition~\ref{P:visi_horo}. In what follows, we shall use the 
         following expression 
           \[
             \koba_\OM(z, w)-\koba_\OM(o, w)=\koba_\OM(o, z)-2\langle z | w \rangle^{\OM}_{o}. 
             \]
        in the defining inequality of the horospheres. We now present:
          \begin{proof}[The proof of Proposition~\ref{P:visi_horo}] 
          Fix $o\in\OM$ and $x\in\bdy{\OM}$. 
          Since $\OM$ is a visibility domain, there exists a geodesic ray $\gamma:[0, \infty)\lrarw\OM$
          that emanates from $o$ and lands at $x$, i.e., 
          \[
            \gamma(0)=o \ \ \ \text{and} \ \ \ \lim_{t\to\infty}\gamma(t)=x.
            \]
            The above result follows from \cite[Lemma~2.16]{CGMS:2024}; for bounded complete hyperbolic domains it was 
            established in \cite[Lemma~3.1]{BNT:2022}. Then by Lemma~\ref{lm:clusgeodray} for any $R>0$, we have 
            $x\in \clos{H_o^b(x, R)}$. 
            \smallskip
         
           Suppose, to get a contradiction, that there exists a $y\in\bdy{\OM}$ such that $y\neq x$ and
           $y\in \overline{H_o^b(x, R)}$ for some $R>0$. By definition, there exists a 
          sequence $(y_n)\subset H_o^b(x, R)$ such that $y_n\to y$ as $n\to\infty$. Note, this implies, for every $y_n$ we have
          \begin{equation}\label{E:seqinhoro}
           \limsup_{w\to x}\big(2\langle y_n | w \rangle^{\OM}_{o}-\koba_\OM(o, y_n)\big)>({-1}/{2})\log R. 
           \end{equation}
           Since $\OM$ is a visibility domain, there exists a positive number $M$ such that 
           \[ 
           \limsup_{z\to y,\,w\to x}\langle z | w \rangle^{\OM}_{o} < M. 
           \]
           The above inequality follows from \cite[Proposition~2.4]{BNT:2022}. 
           This implies there exists $n_0\in\nat$ and an $r>0$ such that
           \[
            \sup_{w\in B(x, r)\cap\OM}\langle y_n | w \rangle^{\OM}_{o}< M \ \ \ \forall n\geq n_0. 
            \]
            Therefore, for all $n\geq n_0$, we have 
            \begin{equation}\label{E:gromfin}
             \sup_{w\in B(x, r)\cap\OM}\big(2\langle y_n | w \rangle^{\OM}_{o}-\koba_\OM(o, y_n)\big)<2M-\koba_\OM(o, y_n). 
             \end{equation}
             From \eqref{E:seqinhoro} and \eqref{E:gromfin} we get that 
             \[
              2M-\koba_\OM(o, y_n)>({-1}/{2})\log R\Longleftrightarrow\koba_\OM(o, y_n)<2M+({1}/{2})\log R \ \ \forall n\geq n_0. 
              \]
              But this is a contradiction to the fact that $\lim_{n\to\infty}\koba_\OM(o, y_n)=\infty$.
              Therefore, no point of $\bdy\OM$
              other than $x$ is in the set $\overline{H_o^b(x, R)}\cap\bdy{\OM}$ whence the result. 
              \end{proof}
              
              \section{The proofs of Theorem~\ref{thm:gromovmodel}, Theorem~\ref{thm:locmodgrom} Theorem~\ref{thm:loccondini}}\label{S:proofgromovmodel}
              %In this section, we shall present the proofs of Theorem~\ref{thm:gromovmodel} and Theorem~\ref{thm:dinismooth}.
              \subsection{The proof of Theorem~\ref{thm:gromovmodel}}\label{ss:proofgromovmodel}
              \begin{proof}
              Fix a point $x\in\bdy\OM$ and let $o\in\OM$ and $R<1$ be given. 
              By \cite[Theorem~1.11]{BZ:2023}, we know that $(\OM, \koba_\OM)$ is Cauchy-complete and has visibility property 
              with respect to the geodesics of the Kobayashi distance. Therefore, by Proposition~\ref{P:visi_horo}, we have 
              \[  
               \overline{H_o^s(x, R)}\cap\bdy\OM\subset\overline{H_o^b(x, R)}\cap\bdy{\OM}=\{x\}. 
               \]
              We shall now show that $x\in\overline{H_o^s(x, R)}\cap\bdy\OM$. Using visibility property, choose a geodesic ray $\beta:[0, \infty)\lrarw\OM$
              that starts from the point $o$ and lands at the point 
              $x$, i.e. $\beta(0)=o$, $\lim_{t\to\infty}\beta(t)=x$. Choose $T$ large enough such that 
              \begin{equation}\label{E:Tvalue}
                -T+10\delta<(1/2)\log R,
                \end{equation}
               where $\delta\geq 0$ is such that $(\OM, \koba_\OM)$ is $\delta$-hyperbolic. Consider the point $z_T:=\beta(T)$. 
               \smallskip 
               
               \noindent{\bf Claim.} For any $T$ that satisfies \eqref{E:Tvalue}, $z_T\in H^s_o(x, R)$. 
               \smallskip 
               
               \noindent Fix a $T$ satisfying \eqref{E:Tvalue}, and suppose, to get a contradiction, that the above claim is not true.
               Then we can find a sequence $(w_n)\subset\OM$
               such that $w_n\to x$, as $n\to\infty$, and
               \begin{equation}\label{E:notinhoro} 
                \lim_{n\to\infty}\big(\koba_\OM(z_T, w_n)-\koba_\OM(o, w_n)\big)\geq(1/2)\log R. 
                \end{equation}
                Let $\gamma_n:[0, L_n]\lrarw\OM$ be a geodesic segment joining $o$ and $w_n$ and let 
                $\sigma_n:[0, M_n]\lrarw\OM$ be a geodesic segment joining $z_T$ and $w_n$. Appealing to 
                Ascoli--Arzel\'{a} Theorem, and passing to subsequences, without loss of 
                any generality, we can assume that $\gamma_n$, $\sigma_n$ converges to geodesic rays $\gamma$ and 
                $\sigma$, respectively, locally uniformly on $[0, \infty)$. Since $\OM$ is a visibility domain, 
                by \cite[Lemma~3.1]{BNT:2022}, both $\gamma$ and $\sigma$ 
                land at the point $x$. By our assumption, $x$ represents a unique point in the Gromov 
                boundary $\bdy_G\OM$, therefore both $\gamma$ and $\beta$ represent the same 
                point in the Gromov boundary as represented by $x$. 
                Since $\OM$ is $\delta$-hyperbolic,
                using \cite[Lemma~3.3, Chapter~III.H]{BH:1999}, we have 
                \[
                 \koba_\OM(\gamma(t), \beta(t))\leq 2\delta \ \ \ \forall t\geq 0.
                 \]
                 Consider $\tilde{\beta}:[0, \infty)\lrarw\OM$ defined by $\tilde{\beta}(t):=\beta(t+T)$ for all $t\geq 0$. Note that
                 $\tilde{\beta}$ is a geodesic ray that starts at $z_T$ and lands at $x$. Therefore, by a similar 
                 argument as above, we get
                 \begin{equation}\label{E:sigmabetaT}
                   \koba_\OM(\sigma(t), \beta(t+T))\leq 2\delta \ \ \ \forall t\geq 0.
                   \end{equation}
                   From the above two inequalities we get: 
                   \begin{align}\label{E:gammasigmaT}
                    \koba_\OM(\sigma(t), \gamma(t+T))\leq \koba_\OM(\sigma(t), \beta(t+T))+\koba_\OM(\beta(t+T), \gamma(t+T))
                                                                             \leq 4\delta \ \ \ \forall t\geq 0. 
                                                                             \end{align}
                  Fix a $t_0\geq 0$ and write 
                  \[ 
                   \koba_\OM(z_T, w_n)=\koba_\OM(z_T, \sigma_n(t_0))+\koba_\OM(\sigma_n(t_0), w_n). 
                   \]  
                   Now, for large $n$, we write 
                   \begin{align*} 
                   \koba_\OM(o, w_n)&=\koba_\OM(o, \gamma_n(t_0+T))+\koba_\OM(\gamma_n(t_0+T), w_n)\\
                                                 &\geq\koba_\OM(o, \sigma_n(t_0))-\koba_\OM(\sigma_n(t_0), \gamma_n(t_0+T))+
                                                     \koba_\OM(w_n, \sigma_n(t_0))-\koba_\OM(\sigma_n(t_0), \gamma_n(t_0+T))\\
                                                 &= \koba_\OM(o, \sigma_n(t_0))-2\koba_\OM(\sigma_n(t_0), \gamma_n(t_0+T))+
                                                     \koba_\OM(w_n, \sigma_n(t_0)). 
                 \end{align*}
                 From the last two inequalities, for large enough $n$, we get
                 \[
                  \koba_\OM(z_T, w_n)-\koba_\OM(o, w_n)\leq\koba_\OM(z_T, \sigma_n(t_0))- \koba_\OM(o, \sigma_n(t_0))+ 
                                                                                           2\koba_\OM(\sigma_n(t_0), \gamma_n(t_0+T)).
                                                                                           \]
                 Taking limit, as $n\to\infty$, and using inequality \eqref{E:gammasigmaT}, we get 
                 \begin{align}\label{E:crucial}
                    \lim_{n\to\infty}\big(\koba_\OM(z_T, w_n)-\koba_\OM(o, w_n)\big)&\leq\koba_\OM(z_T, \sigma(t_0))- \koba_\OM(o, \sigma(t_0))+ 
                                                                                                                                2\koba_\OM(\sigma(t_0), \gamma(t_0+T))\nonumber \\
                                                                                                                                &\leq t_0-\koba_\OM(o, \sigma(t_0))+8\delta
                                                                                           \end{align}
                 
                 Appealing to \eqref{E:sigmabetaT}, we can write
                 \begin{align*}
                   \koba_\OM(o, \sigma(t_0))&\geq\koba_\OM(o, \beta(t_0+T))-\koba_\OM(\beta(t_0+T), \sigma(t_0))\\
                                                              & \geq t_0+T-2\delta
                                                              \end{align*}
                 
                 Putting this into \eqref{E:crucial} and by our choice of $T$, we get 
                 \begin{align*} 
                  \lim_{n\to\infty}\big(\koba_\OM(z_T, w_n)-\koba_\OM(o, w_n)\big)&\leq t_0-(t_0+T-2\delta)+8\delta\\
                                                                                                                           &= -T+10\delta\\
                                                                                                                           & <(1/2)\log R,
                                                                                                                           \end{align*}
                   which is a contradiction to the inequality \ref{E:notinhoro} whence the claim.
                   \hfill$\blacktriangleleft$
                   \smallskip
                   
                   Therefore, given $x\in\bdy\OM$, $o\in\OM$ and $R<1$, and a geodesic ray $\beta:[0, \infty)\lrarw\OM$ that starts 
                   from $o$ and lands at $x$, for any $T>10\delta-(1/2)\log R$, the point $z_T=\beta(T)$ belongs to $H^s_o(x, R)$. 
                   In particular, $x\in\clos{H^s_o(x, R)}\cap\bdy\OM$. This establishes the result. 
              \end{proof}
We now present an alternative proof of Theorem~\ref{thm:gromovmodel}\,---\,as
alluded to in Remark~\ref{Rem:altproof}. However, 
we first recall a few definitions, and state two results that are at the heart of this proof. 
\smallskip

Given $(\OM, \koba_\OM)$, a complete Kobayashi hyperbolic domain, that is Gromov hyperbolic, we 
shall denote by $\clos{\OM}^H$, $\bdy\clos{\OM}^H$ the horofunction compactification, and the horofunction
boundary respectively. We refer the reader to \cite[Section~2]{AFSG:2024} for a brief introduction in our setting to 
these objects. Given $\xi\in\bdy\clos{\OM}^H$, and $p\in\OM$, by definition, there exists a unique 
continuous function $h_{p,\,\xi}$, vanishing at $p$, that represents $\xi$.
Given $R>0$, the set
\[
H_p(\xi, R):=\left\{x\in X\,:\,h_{p,\,\xi}(x)<\tfrac{1}{2}\log R\right\}
\]
is called the horoball of radius $R$ centered at $\xi$ with pole $p$. 
It is known fact \cite[Proposition~4.6]{WW:2005} that there exists a unique 
$\Phi:\clos{\OM}^H\lrarw\clos{\OM}^G$ extending the identity 
map continuously on $\OM$. Moreover, we have the following results:
\begin{result}[{paraphrasing \cite[Proposition~4.4]{WW:2005}}]\label{Res:unifboun}
Let $\xi_1$, $\xi_2\in\bdy\clos{\OM}^H$ be such that $\Phi(\xi_1)=\Phi(\xi_2)$.
Then, for any $p\in\OM$,
\[
\big|h_{p,\,\xi_1}(z)-h_{p,\,\xi_2}(z)\big|\leq 2\delta\quad \text{for all}\, z\in\OM, 
\]
where $\delta\geq 0$ is such that $(\OM, \koba_\OM)$ is $\delta$-hyperbolic.
\end{result}

\begin{result}[{paraphrasing \cite[Proposition~6.5]{AFSG:2024}}]\label{res:unihoroball}
Given $\xi\in\bdy\clos{\OM}^H$, $p\in\OM$, for any $R>0$, 
\[
\clos{H_{p}(\xi, R)}^{G}\cap\bdy_G{\OM}=\{\Phi(\xi)\}.
\]
Here, the closure at the left-hand side is in the Gromov topology. 
\end{result}

%As an application of 
%the above two results, we present the proof alluded to in Remark~\ref{Rem:altproof}.
\begin{proof}
Fix a point $o\in\OM$, $x\in\bdy\OM$ and let $R>0$ be given. Note 
$x$ represents a unique point in $\bdy_G\OM$. 
We claim that there exists an $R'>0$ such that 
\[
H_o(\xi, R')\subset H^b_o(x, R')\subset H^s_o(x, R)\subset H_o(\xi, R)
\quad \forall\xi\in\Phi^{-1}(x). 
\]
To see the claim, note that it is a fact that for any $S>0$
\[
H^b_o(x, S)=\bigcup_{\xi\in\Phi^{-1}(x)}H_o(\xi, S),
\quad H^s_o(x, S)=\bigcap_{\xi\in\Phi^{-1}(x)}H_o(\xi, S). 
\]
Therefore, we only need to show the second inclusion in the claim above. Let $R'>0$ be such that 
$(1/2)\log{R'}+2\delta\leq {(1/2)}\log{R}$. Let $z\in H^b_o(x, R')$, then
there exists $\xi\in\Phi^{-1}(x)$ such that $z\in H_o(\xi, R')$.
Now if $\zeta\in\Phi^{-1}(x)$ be 
any other point, then using Result~\ref{Res:unifboun} we get 
\[
h_{o,\,\zeta}(z)\leq 2\delta + h_{o,\,\xi}(z)< 2\delta + (1/2)\log R'\leq (1/2)\log R. 
\]
Therefore, for any $\zeta\in\Phi^{-1}(x)$, we have $z\in H_p(\zeta, R)$
whence $H^b_o(x, R')\subset H^s_o(x, R)$. This establishes the claim. 
\smallskip

The claim above together with Result~\ref{res:unihoroball} implies that 
$\clos{H^s_o(x, R)}^{G}\cap\bdy_G\OM=\{x\}$. For big horospheres, one could change the 
role of $R, R'$ in the claim above, and then the similar argument works. 
\end{proof}

              \subsection{The proof of Theorem~\ref{thm:locmodgrom}}\label{ss:proof-loc-mod-grom}
              We shall now present the proof of Theorem~\ref{thm:locmodgrom}, but first we recall:
              given $\OM$, a complete Kobayashi hyperbolic domain, we say that $\OM$ is a local model-Gromov-hyperbolic 
              domain if for any point $x\in\bdy\OM$ there is a neighbourhood $U_x$ in $\C^d$
              such that $\OM\cap U_x$ is connected 
              and $(\OM\cap U_x, \koba_{\OM\cap U_x})$ is a model-Gromov-hyperbolic domain. 
              \begin{proof}[The proof of Theorem~\ref{thm:locmodgrom}]
             Fix an $x\in\bdy{\OM}$, then there exists a neighbourhood $U_x$ of $x$
             in $\C^d$ such that $(\OM\cap U_x, \koba_{\OM\cap U_x})$ is a model-Gromov-hyperbolic domain. 
             Denote by $V_x:=U_x\cap\OM$. 
             By \cite[Theorem~1.11]{BZ:2023}, it follows that $V_x$ has the weak visibility property 
             with respect to $\koba_{V_x}$. Note that since $x\in\bdy\OM$ is arbitrary, and $\OM$ is hyperbolically
             embedded in $\C^d$, appealing to
             Result~\ref{res:locvis-globvis}, we get that $\OM$ is a weak visibility domain. 
             Since $\OM$ is complete Kobayashi hyperbolic, by Result~\ref{res:vis-weakvis}, $\OM$ also 
             has visibility property with respect to the geodesics 
             of the Kobayashi distance. Therefore, by Proposition~\ref{P:visi_horo}, we have 
              \[  
               \overline{H_o^s(x, R)}\cap\bdy\OM\subset\overline{H_o^b(x, R)}\cap\bdy{\OM}=\{x\}. 
               \]
               We shall now show that $x\in \overline{H_o^s(x, R)}\cap\bdy\OM$ for all $R>0$. 
               Given $x\in\bdy\OM$, we choose $r$ small enough such that $B(x, r)\Subset U_x$.
               Let $W_x:=B(x, r)\cap\OM$. 
                Since $\OM$ is a visibility domain, by Result~\ref{res:localization}, there exists a $C>0$ 
                such that 
                \[
                 \koba_\OM(z, w)\leq\koba_{V_x}(z, w)\leq\koba_\OM(z, w)+C \ \ \ \forall z, w\in W_x. 
                 \]
                 Fix a point $p\in W_x$, then for a point $z\in W_x$, we have
                 \[
                  \koba_\OM(z, w)-\koba_\OM(p, w)\leq\koba_{V_x}(z, w)-\koba_{V_x}(p, w)+C, \ \ \ \forall w\in W_x. 
                  \]
                  The above inequality then implies that for $z, p\in W_x$ we have
                  \[
                    \limsup_{w\to x}\big(\koba_\OM(z, w)-\koba_\OM(p, w)\big)
                      \leq\limsup_{w\to x}\big(\koba_{V_x}(z, w)-\koba_{V_x}(p, w)\big)+C. 
                      \]
                      Therefore, it follows that 
                      \begin{equation}\label{E:lochoro-globhoro}
                       W_x\cap H^s_p(x, R, {V_x})\subset H^s_p(x, Re^{2C}, \OM),
                       \end{equation}
                       where $p\in W_x$ is fixed. Now, since $({V_x}, \koba_{V_x})$ is a model-Gromov-hyperbolic domain
                       by Theorem~\ref{thm:gromovmodel} we have 
                       \[ 
                        x\in\clos{H^s_p(x, R, {V_x})} \ \ \ \forall R>0. 
                        \]
                        Therefore, it follows from \eqref{E:lochoro-globhoro} that for all $R>0$, we have
                        \[
                         x\in\clos{H^s_p(x, Re^{2C}, \OM)}. 
                         \]
                         Since $R$ is arbitrary, it follows that $x\in H^s_p(x, R)$ for all $R>0$ and then arguing as at the end of the proof of Lemma~\ref{lm:clusgeodray}, we have $x\in H^s_o(x, R)$ for all $R>0$, $o\in\OM$ and $x\in\bdy{\OM}$. 
                         This establishes the result. 
                         \end{proof}
                               
\subsection{The proof of Theorem~\ref{thm:loccondini}}\label{ss:dinismooth}
In this subsection, we shall present the proof of Theorem~\ref{thm:loccondini} but first we need a lemma.
\begin{lemma}\label{lm:dinismooth}
Let $\OM$ be a bounded convex domain and let $x\in\bdy\OM$ be a Dini-smooth boundary point. 
Then for any $R>0$ and $o\in\OM$ we have
\[
 x \in \clos{H^s_o(x, R)}.
  \]
  \end{lemma}
\begin{proof}
Since $\OM$ is a bounded convex domain, 
by \cite[Proposition~2.4]{Mercer:1993}, there exists a $C\in\R$ such that for every $w\in\OM$, we have
\[ 
 \koba_\OM(o, w)\geq C+\frac{1}{2}\log\frac{1}{\delta_\OM(w)},
  \]
where $\delta_\OM(\cdot)$ denotes the boundary distance function. 
Since $x\in\bdy{\OM}$ is a Dini-smooth boundary point, 
by a result of Nikolov--Andreev \cite{NA:2017}, there exists $\epsilon>0$, $c>0$, such that for every
$z, w\in B(x, \epsilon)$, we have 
\[ 
 \koba_\OM(z, w)\leq\log\Bigg(1+\frac{c\|z-w\|}{\sqrt{\delta_\OM(z)\delta_\OM(w)}}\Bigg).
  \]
From the above two inequalities, for any $z, w\in B(x, \epsilon)$ we get
\[
 \koba_\OM(z, w)-\koba_\OM(o, w)\leq\log\Bigg(\sqrt{\delta_\OM(w)}+\frac{c\|z-w\|}
  {\sqrt{\delta_\OM(z)}}\Bigg)-C. 
    \]
Therefore, for a fixed $z\in B(x, \epsilon)$, we have 
\begin{align*}
\limsup_{w\to x}\big(\koba_\OM(z, w)-\koba_\OM(o, w)\big)&\leq\log\Bigg(\frac{c\|z-x\|}
                       {\sqrt{\delta_\OM(z)}}\Bigg)-C\\
                       &=\frac{1}{2}\log{\|z-x\|}+\frac{1}{2}\log\frac{\|z-x\|}{\delta_\OM(z)}+(\log c-C).
                        \end{align*}
                        If we choose $z$ on the normal line at the point $x$, and $z$ is sufficiently close to $x$,
                        then the second term in the inequality above remains 
                        bounded above by a constant. This shows 
                        for any $R<1$, we can find an $s>0$ such that for any point $z$ that lies on the normal line 
                        passing through $x$ with 
                        $\delta_\OM(z)<s$, we have 
                        \[
                         \limsup_{w\to x}\big(\koba_\OM(z, w)-\koba_\OM(o, w)\big)<\frac{1}{2}\log R. 
                         \]
                         From this it follows that for any $R<1$, and hence for any $R>0$, 
                         the set $\clos{H^s_o(x, R)}\cap\bdy{\OM}$
                         contains $x$. This together with our observation at the start establishes the result.   
                   \end{proof}
We now present 
\begin{proof}[The proof of \ref{thm:loccondini}]
Since $\Phi_x(V_x\cap\OM)$ has visibility property with respect to the geodesics and $\Phi_x$ is a biholomorphism, 
it is not difficult to check that 
$V_x\cap\OM$ also has the visibility property with respect to the geodesics of $\koba_{V_x\cap\OM}$.
By Result~\ref{res:vis-weakvis}, $(V_x\cap\OM, \koba_{V_x\cap\OM})$ is a weak visibility domain. 
Since every bounded domain is hyperbolically imbedded in $\C^d$, appealing to Result~\ref{res:locvis-globvis},
we deduce that $(\OM, \koba_\OM)$ is a weak visibility 
domain. Since $\koba_\OM$ is complete, it satisfies visibility property with 
respect to the geodesics. Therefore, by Proposition~\ref{P:visi_horo}, given $o\in\OM$, $x\in\bdy\OM$ 
and $R>0$ we have 
\[
 \clos{H^s_o(x, R)}\cap\bdy{\OM}\subset\clos{H^b_o(x, R)}\cap\bdy{\OM}=\{x\}.
 \]
 Fix a point $x\in\bdy\OM$, to show that $x\in\clos{H^s_o(x, R)}$, consider 
 $(\Phi_x, U_x, V_x)$ as in the definition of local convexity. Then, for very small $r$,
 by the localization result, there exists a $C>0$ such that 
 \[
   \koba_\OM(z, w)\leq\koba_{V_x\cap\OM}(z, w)\leq\koba_\OM(z, w) + C \ \ \ \forall \ z, w\in B(x, r)\cap\OM. 
   \]
   Write $W_x:=B(x, r)\cap\OM$. Fix a point $p\in W_x$, then for a point $z\in W_x$, we have
                 \[
                  \koba_\OM(z, w)-\koba_\OM(p, w)\leq\koba_{V_x\cap\OM}(z, w)-\koba_{V_x\cap\OM}(p, w)+C, \ \ \ \forall w\in W_x. 
                  \]
                  The above inequality then implies that for $z, p\in W_x$ we have
                  \[
                    \limsup_{w\to x}\big(\koba_\OM(z, w)-\koba_\OM(p, w)\big)
                      \leq\limsup_{w\to x}\big(\koba_{V_x\cap\OM}(z, w)-\koba_{V_x\cap\OM}(p, w)\big)+C. 
                      \]
                      Therefore it follows that 
                      \begin{equation}\label{E:loccontglob}
                       W_x\cap H^s_p(x, R, {V_x}\cap\OM)\subset H^s_p(x, Re^{2C}, \OM),
                       \end{equation}
                       where $p\in W_x$ is fixed. Note that $\Phi_x(V_x\cap\OM)$ is a bounded convex domain and $\Phi_x(x)$ is a Dini-smooth boundary point. Therefore, by Lemma~\ref{lm:dinismooth}, 
                       \[
                        \Phi_x(x)\in\clos{H^s_{\Phi_x(p)}(\Phi_x(x), R, \Phi_x({V_x}\cap\OM))}. 
                        \]
                        Since $\Phi_x:V_x\cap\OM\lrarw\Phi_x({V_x}\cap\OM)$ is a biholomorphism
                        that extends as a homeomorphism, we get that 
                        \[ 
                        x\in\clos{H^s_p(x, R, {V_x}\cap\OM)} \ \ \ \forall R>0. 
                        \]
                        Therefore, it follows from \eqref{E:loccontglob} that for all $R>0$, we have
                        \[
                         x\in\clos{H^s_p(x, Re^{2C}, \OM)}. 
                         \]
Since $R$ is arbitrary, it follows that $x\in H^s_p(x, R, \OM)$ for all $R>0$ and then arguing as at the end 
of the proof of Lemma~\ref{lm:clusgeodray}, we have $x\in H^s_o(x, R, \OM)$ for all $R>0$, $o\in\OM$ and $x\in\bdy{\OM}$. 
This establishes the result.                                                         
\end{proof}                         
                   
\section{Continuous extension of biholomorphisms}\label{S:context}
In this section, we present the proof of our result regarding the continuous extension
of biholomorphisms. As an upshot, we shall characterize those bounded simply-connected domains
in $\C$ that have property \eqref{E:smallbigsingle} in problem $(*)$.
We first recall a Julia lemma essentially due to Abate.
\begin{result}\label{Res:Julialemma}
Let $D_1\subset\C^n$ and $D_2\subset\C^m$ be bounded domains, and let 
$f:D_1\to D_2$ be a holomorphic map. 
Suppose there exist points $p\in D_1$, $q\in D_2$, a constant $\alpha>0$, 
and a sequence $(w_\nu)\subset D_1$ converging to $x\in\partial D_1$ such that 
$f(w_\nu)\to y\in\partial D_2$ and
\[
\lim_{\nu\to\infty}\!\big(\koba_{D_1}(p,w_\nu)
    -\koba_{D_2}(q,f(w_\nu))\big)
   \le \tfrac12\log\alpha.
\]
Then, for every $R>0$,
\[
f\big(H^s_p(x,R)\big)\subset H^b_q(y,\alpha R).
\]
\end{result}
\noindent Abate proved the above result when $D_1=D_2$ and $p=q$
(see, e.g., \cite[Proposition~2.4.15]{{Abate:iteration89}}). The general
case follows by the same argument.

\begin{lemma}\label{lm:contextprop}
Let $D$, $\OM$ be bounded domains that are complete with respect to the Kobayashi distance. Suppose
there exist $o\in D$, $p\in\OM$ such that
\begin{itemize}
\item for any two points $y_1\neq y_2\in\bdy D$, there exists an $R>0$
such that $H^b_o(y_1, R)\cap H^b_o(y_2, R)=\emptyset$, and
\smallskip

\item for any $x\in\bdy\OM$ and any $R>0$, $H^s_p(x, R)\neq\emptyset$.
\end{itemize}
Then any biholomorphism $\Phi:\OM\to D$ extends as a continuous map from $\clos{\OM}$ onto $\clos{D}$.
\end{lemma}

\begin{proof} 
Let $o\in D$ be as in the statement of the lemma. Observe that\,---\,owing to the last part
of Result~\ref{res:horo_prop}\,---\,we may assume that
$p\in\OM$ satisfies $\Phi(p)=o$. 
We claim that, given $x\in\bdy\OM$, let $(w_n)\subset\OM$ be a sequence
that converges to $x$ such that 
$\Phi(w_n)\to y\in\bdy D$. Then
\begin{equation}\label{E:smallintobig}
 \Phi(H^s_p(x, R))\subset H^b_o(y, R) \ \ \ \forall R>0. 
 \end{equation}
 To see this, we note that\,---\,since $\Phi$ is an isometry\,---\,for each $n$, we have
 \begin{align*}
 \koba_\OM(p, w_n)-\koba_D(o, \Phi(w_n))= \koba_\OM(p, w_n)-\koba_D(\Phi(p), \Phi(w_n))
   = \koba_\OM(p, w_n)-\koba_D(p, w_n)=0.
 \end{align*}
 Therefore, an application of Result~\ref{Res:Julialemma} with $\alpha=1$,
 establishes \eqref{E:smallintobig}.
 \smallskip
 
 Fix a point $x\in\bdy\OM$, we shall show that the limit 
  \[
   \lim_{z\in\OM, \,z\to x}\Phi(z)
  \]
exists. Suppose, to get a contradiction, that the above limit
does not exist. This implies that\,---\,since $\Phi$ is a biholomorphism and $D$ is bounded\,---\,there exist sequences $(z_n)$ and $(w_n)$ in 
$\OM$ that converge to $x$ and points $y_1\neq y_2\in\bdy D$
such that $\Phi(z_n)\to y_1$, $\Phi(w_n)\to y_2$. 
Applying \eqref{E:smallintobig}, we get 
\[  
  \Phi(H^s_o(x, R))\subset H^b_o(y_1, R)\cap H^b_o(y_2, R) \ \ \ \forall R>0. 
  \]
Since $y_1$ and $y_2$ are distinct, there exists 
%by the geometry of the horospheres in the unit disc, we can choose 
an $R_0>0$ such that $H^b_o(y_1, R)\cap H^b_o(y_2, R)=\emptyset$ for all $R<R_0$. This implies that 
$H^s_o(x, R)=\emptyset$ for all $R<R_0$ which is a contradiction to our hypothesis whence the 
result. 
\smallskip

It is now routine to check that $\Phi$ extends continuously on $\clos{\OM}$.
\end{proof}

We now present

\begin{proof}[The proof of Theorem~\ref{thm:contextbiholo}]
Since $D$ is metrically-regular, there exists a point $o\in D$ such that $(1)$ and $(2)$ 
in Definition~\ref{D:metreg} are satisfied.  
Let $\Phi:\OM\lrarw D$ be the inverse of the
map $\Psi$ and let $p\in\OM$ be such that $\Phi(p)=o$.
\smallskip

Lemma~\ref{lm:contextprop} implies that $\Phi:\OM\lrarw D$ extends as a continuous map from 
$\clos{\OM}$ onto $\clos{D}$\,---\,that we continue to denote as $\Phi$ itself. Consider a point 
$x\in\bdy\OM$ and let $\xi=\Phi(x)$. Given a point $z\in H^b_p(x, R)$,  
there exists a sequence $(z_n)$ in $\OM$ such that $z_n\to x$ and 
 \[
  \liminf_{w\to x}\big(\koba_\OM(z, w)-\koba_\OM(p, w)\big)=
   \lim_{n\to\infty}\big(\koba_\OM(z, z_n)-\koba_\OM(p, z_n)\big) <(1/2)\log R. 
  \]
  Since $\Phi$ is an isometry with respect to the Kobayashi distances, from the above inequality, we get
  \[
  \lim_{n\to\infty}\big(\koba_D(\Phi(z), \Phi(z_n))-\koba_D(o, \Phi(z_n))\big)<(1/2)\log R.
  \]
  Note $\Phi$ is continuous at $x$ and $\Phi(x)=\xi$, since $(z_n)$ converges to $x$, it follows that $(\Phi(z_n))$
  converges to $\xi$. 
  This implies that $\Phi(z)\in H(\xi, R)$ for any $z\in H^b_p(x, R)$. Therefore,
  \begin{equation}\label{E:phibighorounithoro}
   \Phi(H^b_p(x, R))\subset H(\Phi(x), R) \ \ \ \forall R>0.
   \end{equation}

   Since small horospheres and big horospheres in $D$ are identical, and 
   $\Psi$ an isometry, an application of Result~\ref{Res:Julialemma}\,---\,as in
   the proof of Lemma~\ref{lm:contextprop}\,---\,implies that
   given $\zeta\in\bdy D$, and $(\zeta_n)\subset D$, a sequence
   converging to $\zeta$, such that $\Psi(\zeta_n)\to y$ for some $y\in\bdy\OM$,
   we have
\begin{equation}\label{E:psiunithorobighoro}
 \Psi(H(\zeta, R)\subset H^b_p(y, R) \ \ \ \forall R>0. 
 \end{equation}
Fix a point $\xi\in\bdy D$, we shall show that the limit 
 \[
   \lim_{z\in D, \,z\to\xi}\Psi(z) \ \  \text{exists}. 
   \]
Suppose, to get a contradiction, that the above is not true. This implies
that\,---\,since $\Psi$ is a biholomorphism and $\OM$ is bounded\,---\,there exist sequences
$(z_n), (w_n)\subset D$ converging to $\xi$ and $x_1\neq x_2\in\bdy\OM$ 
such that $\Psi(z_n)\to x_1$ and $\Psi(w_n)\to x_2$. Applying
\ref{E:psiunithorobighoro}, we get
\[
 \Psi(H(\xi, R))\subset H^b_p(x_i, R) \ \ \ \forall R>0, \ i=1, 2. 
 \]
 Since $\Phi$ is continuous at $x_1$, the sequence $\Phi(\Psi(z_n))=z_n$ converges to $\Phi(x_1)$, as $n\to\infty$. 
 This implies that $\Phi(x_1)=\xi$. Similarly 
 $\Phi(x_2)=\xi$. By \eqref{E:phibighorounithoro}, we have
 \[
  \Phi(H^b_p(x_i, R))\subset H(\xi, R) \ \ \ \forall R>0,\,i=1,2. 
  \]
  From the last two relations, for each $i=1, 2$, we get
  \[
   H^b_p(x_i, R)\subset\Psi(H(\xi, R))\subset H^b_p(x_i, R)\ \ \ \forall R>0. 
   \]
This implies that $H^b_p(x_1, R)=H^b_p(x_2, R)=\Psi(H(\xi, R))$ for all $R>0$. But this contradicts 
our hypothesis regarding the big horospheres on the domain $\OM$.
\smallskip
   
This implies that $\Psi$ extends as a continuous map from $\clos{D}$ onto
$\clos{\OM}$. Therefore $\Psi:D\lrarw\OM$ extends as a homeomorphism from  $\clos{D}$
onto $\clos{\OM}$. 
 \end{proof}
 
As a corollary of the above result, we present:
\begin{proof}[The Proof of Theroem~\ref{thm:jordandomain}]
Recall $\OM$ is a bounded simply-connected domain that satisfies 
\[
\clos{H^s_o(x, R)}\cap\bdy\OM=\clos{H^b_o(x, R)}\cap\bdy\OM=\{x\},
\]
for any $x\in\bdy\OM$ and for any $R>0$, with $o=\psi(0)$, where $\psi:\unitdisk\lrarw\OM$
is a Riemann map. Clearly $H^s_o(x, R)$ is nonempty for all $x\in\bdy\OM$, $R>0$ and 
also given $x\neq y\in\bdy\OM$, $H^b_o(x, R)\neq H^b_o(y, R)$ for any $R>0$. Clearly
$\unitdisk$ is metrically-regular, therefore, 
by Theorem~\ref{thm:contextbiholo}\,---\,taking $D=\unitdisk$\,---\,we see that $\psi$ extends 
as a homeomorphism from $\clos{\unitdisk}$ onto $\clos{\OM}$. Hence $\OM$ has to be a Jordan 
domain. 
\end{proof}

We now present an example of a domain that does not satisfy condition $(a)$ in Theorem~\ref{thm:contextbiholo} but satisfies 
condition $(b)$ therein. 
\begin{example}\label{exa:smallhoroempty}
Consider the simply-connected domain
\[
 \OM:=\unitdisk\setminus\{[0, 1)\}.
 \]
Let $\psi:\unitdisk\lrarw\OM$ be a Riemann map and $\phi:\OM\lrarw\unitdisk$ be its inverse. Since $\bdy\OM$ is locally connected, by Carath\'{e}odory's 
Extension Theorem \cite[Theorem~4.3.1]{BCD:2020}, $\psi$ has a continuous extension $\hat{\psi}$ from $\clos{\unitdisk}$ 
onto $\clos{\OM}$. Then, for a given $x\in (0, 1]$, $\hat{\psi}^{-1}(x)$ 
has two pre-images, say $\xi_1\neq\xi_2\in\bdy\unitdisk$, see \cite[Proposition~4.3.5]{BCD:2020}. 
Proceeding exactly as in the claim that appears in the proof of Lemma~\ref{lm:contextprop}, 
we get that 
\[
  \phi\big(H^s_o(x, R)\big)\subset H(\xi_1, R)\cap H(\xi_2, R) \ \ \ \forall R>0, 
  \]
  where $o=\psi(0)$.  As $\xi_1$ and $\xi_2$ are distinct points in $\bdy\unitdisk$, we can find an $R_0>0$ such that for any $R<R_0$, 
$H(\xi_1, R)\cap H(\xi_2, R)=\emptyset$.
Therefore, for each $x\in (0, 1]$, there exists an $R_0>0$ such that 
$H^s_o(x, R)=\emptyset$ for all $R<R_0$. Consequently, $\OM$ does not satisfy condition $(a)$ in Theorem~\ref{thm:contextbiholo}.
\smallskip 

It is a fact that simply-connected visibility domains are exactly those whose boundary is locally 
connected, therefore, $\OM$ is a visibility domain and hence satisfies
\[
\clos{H^b_o(x, R)}\cap\bdy\OM=\{x\} \ \ \ \forall R>0. 
\]
This clearly implies that $\OM$ satisfies condition $(b)$ in Theorem~\ref{thm:contextbiholo}. 
\end{example}

\begin{remark}\label{rm:dichotomy}
Let $\OM$ be a bounded simply-connected domain and let $\phi:\OM\lrarw\unitdisk$ be 
a biholomorphism. Suppose $\OM$ has the property that, given $o\in\OM$, for any $x\in\bdy\OM$, 
and $R>0$, $H^s_o(x, R)$ is nonempty. Then, by Lemma~\ref{lm:contextprop}, $\phi$ extends 
continuously from $\clos{\OM}$ onto $\clos{\unitdisk}$. Then it is a fact that 
$\phi$ has to be a homeomorphism \cite[Theorem~5.12]{Conway} and hence $\OM$ is a Jordan domain. Therefore, 
for a bounded simply-connected domain $\OM$ we have the following dichotomy:
either they are Jordan domains or there exists a point $x\in\bdy\OM$ and an $R>0$
such that $H^s_o(x, R)=\emptyset$ for some $o\in\OM$. 
\end{remark}

\section{Examples}\label{S:examples}
 In one dimension, we introduce a class of domains that are locally
 a model-Gromov-hyperbolic, in particular, they also include
 non-Gromov-hyperbolic domains and domains whose boundary could be a fractal set.          
\begin{definition}\label{d:cond1}
Let $\OM\subset\C$ be a Kobayashi hyperbolic domain. We say 
that $\OM$ satisfies Condition~1 if for any point $p\in\bdy{\OM}$,
there exist an $r>0$ and a topological embedding 
$\tau_p: \overline{\mathbb{D}}\longrightarrow \overline{\Omega}$ such that $\tau_p(\unitdisk)\subset\Omega$ and 
$B(p; r)\cap\OM\subset\tau_p(\unitdisk)$.
\end{definition}
\noindent It is easy to see why domains that satisfy Condition~1 are
local-model-Gromov-hyperbolic domain: take a point $p\in\bdy\OM$,
and let $\tau_p$, $r>0$ be as above in the definition. Note that 
$\tau_p(\unitdisk)$ is a Jordan domain, therefore appealing to 
the Riemann Mapping Theorem and Carath\'{e}odory's Extension Theorem, 
we can assume, without loss of any generality, that $\tau_p$ is a 
biholomorphism that extends as a homeomorphism from $\unitdisk$ to
$\clos{\tau_p(\unitdisk)}$. Clearly $\tau_p(\unitdisk)$ equipped with its 
Kobayashi distance is a model-Gromov-hyperbolic metric space. Taking 
\[
U_p:=\tau_p(\unitdisk)\cup B(p, r)
\]
we see that $U_p\cap\OM=\tau_p(\unitdisk)$, since $p\in\OM$ is arbitrary,
we see that $\OM$ is a local model-Gromov-hyperbolic domain.
\smallskip

\noindent Planar domains that satisfy Condition~1 were recently
introduced in \cite{CGMS:2024} where a characterization of such domains 
is also given: these are domains for which
boundary is locally connected and each boundary component is a 
Jordan curve in the Riemann sphere \cite[Proposition~4.15]{CGMS:2024}.
It was also established 
that these domains satisfy the strongest form of visibility property. 
\begin{example}[{\cite[Proposition~9.1]{BZ:2023}}]\label{exam:plannongrom}
Consider the domain $\OM\subset\C$ defined by 
\[
 \OM:=\C\setminus\cup_{m, n\in\Z}\{z\in\C\,:\,|z-(m+ni)|<1/4\}. 
\]
Clearly $\OM$ satisfies Condition~1 but $(\OM, \koba_\OM)$ is not Gromov hyperbolic. The obstruction to Gromov hyperbolicity comes from the existence of a $\Z^2$ action, see \cite[Proposition~9.1]{BZ:2023} for the details. 
\end{example}

The second example is an example of a domain where each of its boundary component is 
a Jordan curve given locally by the graph of a nowhere differentiable continuous function. It is borrowed
from \cite[Example~6.6]{CGMS:2024}. 
\begin{example}
The {\em Takagi} or {\em blancmange} function is a continuous nowhere-differentiable
function defined as follows.
\begin{equation} \label{eqn:Takagi_fn_def}
	\forall\,t\in\R,\; T(t):=\sum_{j=0}^{\infty} 2^{-j} \distance(2^j t,\Z).
\end{equation}
In the above equation, $\distance(\bcdot\,,\Z)$ denotes the Euclidean distance 
between an arbitrary point of $\R$ and the set of integers $\Z$. 
Note that $T$ is clearly continuous and
$1$-periodic, hence uniformly continuous. It also follows immediately
from the definition that $0\leq T\leq 1$. In \cite{CGMS:2024}, Example~6.6, the domain 
$V_T$ is constructed which is an infinitely connected domain whose boundary components are either
homeomorphic to $\R$ or is Jordan curve that locally is a graph of the Takagi function. It follows that 
$V_T$ is a domain that satisfies Condition~1 but its boundary is very irregular. In fact, no point in 
$\bdy V_T$ is a local Goldilocks point, to use the terminology of \cite[Definition~1.3]{BZ:2023}.
\end{example}\vspace{-5mm}

\section*{Acknowledgments}
The authors are grateful to the anonymous referee for several helpful comments and suggestions,
particularly for pointing out an alternative proof of Theorem~1.3 and for remarks that improved the clarity of Section~5.

\end{document}